\documentclass[10pt,activeacute,amsfonts,reqno]{amsart}
\usepackage{a4wide}
\usepackage[margin=1.0in]{geometry}
\usepackage[english]{babel}
\usepackage{mathabx}
\usepackage{subfigure}
\usepackage{amsmath,amsthm,amsfonts,amssymb,amscd}
\usepackage{amssymb,amsfonts}
\usepackage[all,arc]{xy}
\usepackage{enumerate}
\usepackage{mathrsfs}
\usepackage{marginnote}
\usepackage{mathtools}
\usepackage{caption}
\usepackage{amsthm,thmtools,xcolor}

\usepackage[all]{xy}
\usepackage{tikz-cd}
\usetikzlibrary{cd}
\usepackage[square,sort,comma,numbers,nonamebreak]{natbib}

\theoremstyle{plain}
\newtheorem{thm}{Theorem}[section]

\newtheorem{lem}{Lemma}[section]
\newtheorem{cor}{Corollary}[section]

\newtheorem{lemx}{Lemma}

\theoremstyle{definition}

\theoremstyle{remark}

\numberwithin{equation}{section}
\usepackage{color}

\usepackage[colorlinks = true,
linkcolor = blue,
urlcolor  = blue,
citecolor = blue,
anchorcolor = blue]{hyperref}
\usepackage{hyperref}
\numberwithin{equation}{section}

\DeclareMathOperator{\arcosh}{\mathrm{arcosh}}

\DeclareMathOperator{\sech}{\mathrm{sech}}
\DeclareMathOperator{\spp}{\mathrm{supp}}

\newcommand{\be}{\begin{equation}}
\newcommand{\ee}{\end{equation}}
\newcommand\underrel[3][]{\mathrel{\mathop{#3}\limits_{%
			\ifx c#1\relax\mathclap{#2}\else#2\fi}}}
\providecommand{\abs}[1]{\left\lvert#1\right\rvert}
\providecommand{\norm}[1]{\left\lVert#1\right\rVert}

\title[PCF soliton stability]{Nonlinear Stability of nonsingular solitons of the Principal Chiral Field equation}

\author[M. Alejo]{Miguel \'A. Alejo}
\address{Miguel \'A. Alejo. Departamento de Matem\'aticas, Universidad de C\'ordoba, C\'ordoba, España.}
\email{malejo@uco.es}
\thanks{M.\'A. was supported by Grant PID2022-137228OB-I00 funded by the Spanish Ministerio de Ciencia, Innovaci\'on y Universidades, MICIU/AEI/10.13039/501100011033.}

\author[C.  Mu\~noz]{Claudio Mu\~noz}
\address{Claudio Mu\~noz. Departamento de Ingenier\'ia Matem\'atica and Centro de Modelamiento Matem\'atico (UMI 2807 CNRS), Universidad de Chile, Casilla 170 Correo 3, Santiago, Chile.}
\email{cmunoz@dim.uchile.cl}
\thanks{C.M. was partially funded by Chilean research grants FONDECYT 1191412, 1231250, and Centro de Modelamiento Matem\'atico (CMM), ACE210010 and FB210005, BASAL funds for centers of excellence from ANID-Chile.}

\author[J. Trespalacios]{Jessica Trespalacios}
\address{Jessica Trespalacios. Departamento de Ingenier\'ia Matem\'atica and Centro de Modelamiento Matem\'atico (UMI 2807 CNRS), Universidad de Chile,  Santiago, Chile.}
\email{jtrespalacios@dim.uchile.cl}
\thanks{J.T. was partially funded by the National Agency for Research and Development (ANID)/ DOCTORADO NACIONAL/2019 -21190604, Chilean research grants FONDECYT 1191412, Centro de Modelamiento Matem\'atico (CMM), ACE210010 and FB210005, BASAL funds for centers of excellence from ANID-Chile.}

\subjclass{Primary: 35B35, 35B40, 37K40, 35Q75.}

\keywords{Nonlinear Stability, Principal Chiral Field, Solitons solutions, Virial Identities}	
\date{\today}

\begin{document}
	\begin{abstract}
	We consider the Principal Chiral Field model posed in 1+1 dimensions into the Lie group $\text{SL}(2,\mathbb R)$. In this work, we show the nonlinear stability of small enough nonsingular solitons. The method of proof involves the use of vector field methods, as in a previous work by the second and third authors dealing with Einstein's field equations under the Belinski-Zakharov formalism, extending for all times the size of suitable null weighted norms of the perturbations at time zero. 
	\end{abstract}
	
	\maketitle
	
\section{Introduction}
We are concerned with the Principal Chiral Field model  (PCF) given by
\begin{equation}\label{S:eqPCF}
\partial_t\left(\partial_t g g^{-1}\right) -\partial_x\left(  \partial_x g g^{-1} \right)=0, \quad (t,x) \in \mathbb R \times \mathbb R, 
\end{equation}
valid for a $2\times 2$ Riemannian metric $g$, into the Lie group $\hbox{SL}(2;\mathbb{R})$. Chiral fields on Lie groups represent equivalence classes of the integrable relativistic two-dimensional systems; consequently, \eqref{S:eqPCF} represents an integrable system. Zakharov and Mikhailov \cite{Zakharov1978} showed that classical spinor fields are connected with each such system. Equation \eqref{S:eqPCF} is a nonlinear $\sigma$-model. The first description of the integrability of this model in the language of the commutative representation (\refeq{S:eqPCF}) was given in \cite{Zakharov1979}. Subsequently, different results associated with integrability, conserved quantities and soliton solutions were obtained \cite{Novikov, faddeev1986integrability, beggs1990solitons}, as well as different descriptions of this equation using B\"acklund and Darboux transformations \cite{haider2008, devchand1998hidden}. There are several results associated with the study of the reduction of the PCF equation \eqref{S:eqPCF} in homogeneous spaces of Lie groups. In particular, Zakharov and Mikhailov \cite{Zakharov1978}  studied PCF \eqref{S:eqPCF} for the special unitary group $\hbox{SU}(N)$. In \cite{zakharov1980int}, they studied the connection of this equation with the Nambu-Jona-Lasinian model.  In this work, we follow a different approach; we shall study the PCF in the particular case of the reduction problem on ``symmetric spaces'', following the works \cite{belinskii1978integration, belinski1979stationary, yaronhadad_2013}. The symmetric space considered is the invariant manifold of symmetric matrices lying in the Lie group $\hbox{SL}(2;\mathbb{R})$. This space is not a Lie group, but it can be identified with a hyperboloid in Minkowski spacetime, see \cite{Matzner1967}.

In this work, we shall study the stability of regular soliton solutions of the PCF model. In particular, we are interested in the notion of \emph{orbital and asymptotic stability} of special solutions of PCF model with small initial data perturbations. As far as we know, it seems the first rigorous results in this direction for this kind of solution.

Unlike many previous results related to orbital stability, in this paper we do not follow the classical approach. PCF is a model where standard techniques fail and one needs a new approach. Specifically, we will combine asymptotic stability techniques and conservation of local energy to provide a complete characterization of perturbations of regular soliton solutions of \eqref{S:eqPCF}.

	\subsection{Setting of the model}   Following the same approach as in \cite{JT},  we consider the identification of the PCF equation \eqref{S:eqPCF}  with the following $(1+1)$-dimensional system of the semilinear wave equations
	\begin{equation}\label{S:PCF}
	 	\begin{cases*}
	 		\partial^2_{t}\Lambda- \partial_x^2\Lambda=-2\sinh(2\Lambda)((\partial_x \phi)^2-(\partial_t\phi)^2),\\
	 		\partial_t^2 \phi-\partial_x^2 \phi=-\dfrac{\sinh(2\Lambda)}{\sinh^2(\Lambda)}(\partial_t \phi \partial_t \Lambda -\partial_x \phi \partial_x \Lambda).
	 	\end{cases*}
	 \end{equation}	
This system corresponds to the PCF equation using the so-called Gowdy coordinates (see \cite{MT2023} for full details). 
Here $\Lambda$ is the scalar field that determines the eigenvalues of $g$, and the scalar field $\phi$ describes the deviation of $g$ from being a diagonal matrix. Since $\phi$ is an angle, we take $\phi\in [0,2\pi]$ w.l.g. Therefore $\Lambda, \phi$ can be considered as the two degrees of freedom in the symmetric matrix $g$ \cite{JT,yaronhadad_2013}.

The system \eqref{S:PCF} is a set of coupled quasilinear wave equations, with a rich analytical and algebraic structure. Solutions of  \eqref{S:PCF} are invariant under space and time translations. Indeed, for any $t_0, x_0 \in \mathbb{R}, $ $\Lambda(t-t_0,x-x_0), \phi(t-t_0,x-x_0)$ is also a solution.  Low-order conservation laws for the PCF equation \eqref{S:PCF} are the energy
\begin{equation}\label{S:energy}
		E[\Lambda, \Lambda_t,\phi, \phi_t](t) = \int
		\left( \dfrac{1}{2}((\partial_{x}\Lambda)^{2}+(\partial_{t}\Lambda)^{2})+2\sinh^{2}(\Lambda)((\partial_{x}\phi)^{2}+(\partial_{t}\phi)^{2}) \right)(t,x)dx,
	\end{equation}
and the momentum
\begin{equation}\label{S:momentum}
 P[\Lambda, \Lambda_t,\phi, \phi_t](t) =\int \left( \partial_x\Lambda\partial_t\Lambda + 4\sinh^{2}(\Lambda)\partial_x\phi\partial_t\phi \right)(t,x)dx.
\end{equation} 
Note that the energy is well-defined if $(\Lambda,\partial_t\Lambda)\in \dot H^1 \times L^2$, but a suitable space for $(\phi,\partial_t\phi)$ strongly depends on the weight $\sinh^{2}{\Lambda}$, which can grow exponentially in space, since $ \dot H^1$ can easily contain unbounded functions. In this regard, making sense of the energy (even for classical solutions such as solitons) is subtle and requires a deep and careful analysis.

\subsubsection{Local and global solutions} 
Among the results presented in \cite[Proposition 1.1]{JT}, one has the local existence result for solutions in the energy space. In the analysis of the initial value problem for this system, the regularity of the term $\frac{\sinh(2\Lambda)}{\sinh^2(\Lambda)}$ is delicate, especially when the function $\Lambda(t,x)$ is zero. This case must be carefully analyzed in order to be able to construct a proper local existence result associated to PCF \eqref{S:PCF}. In \cite{JT}, the function  $\Lambda(t,x)$ was set as a perturbation of a constant, nonzero background in the form
	 \begin{equation*}
	 \Lambda(t,x):= \tilde \lambda + \tilde{\Lambda}(t,x), \quad \tilde \lambda \neq 0.
	 \end{equation*}
	 Notice that this choice makes sense with the energy in \eqref{S:energy}, in the sense that $\Lambda\in \dot H^1$ and $\partial_t\Lambda \in L^2$. The basic idea for LWP was to establish some conditions that are required on $\tilde \lambda$ and $\tilde{\Lambda}$ in order to obtain the desired regularity results. With this choice, the system \eqref{S:PCF} can be written in terms of the function $\tilde{\Lambda}(t,x)$ as follows: 
	 \begin{equation}\label{ecuacion_final}
	 	\begin{cases*}
	 		\partial^2_{t}\tilde{\Lambda}- \partial_x^2 \tilde{\Lambda}=-2\sinh(2\tilde \lambda +2\tilde{\Lambda})((\partial_x \phi)^2-(\partial_t\phi)^2),\\
	 		\partial_t^2 \phi-\partial_x^2 \phi =-\dfrac{\sinh(2 \tilde \lambda +2\tilde{\Lambda})}{\sinh^2({\color{black} \tilde \lambda} +\tilde{\Lambda})}(\partial_t \phi \partial_t \tilde{\Lambda} -\partial_x \phi \partial_x \tilde{\Lambda}).
	 	\end{cases*}
	 \end{equation}
Having established the existence of solutions, the second result presented in \cite{JT}  involves whether or not local solutions can be extended globally in time. Actually, one has \cite[ Theorem 1.1]{JT}:
	\begin{thm}\label{S:GLOBAL0}
	Consider the semilinear wave system \eqref{S:PCF} posed in $\mathbb{R}^{1+1}$, with the following initial conditions:
	\begin{equation*}
		\begin{cases}
			(\phi,\tilde{\Lambda})|_{\{t=0\}}= \varepsilon(\phi_0,\tilde{\Lambda}_0), \quad (\phi_0,\tilde{\Lambda}_0)\in C_c^{\infty}(\mathbb{R})^2, \\
			(\partial_t\phi,\partial_t\tilde{\Lambda})|_{\{t=0\}}= \varepsilon(\phi_1,\tilde{\Lambda}_1),\quad   (\phi_1,\tilde{\Lambda}_1) \in C_c^{\infty}(\mathbb{R})^2.  
		\end{cases}
	\end{equation*} 
Then, there exists $\varepsilon_0$ such that if $\varepsilon < \varepsilon_0$, the unique solution remains smooth for all time and has finite conserved energy \eqref{S:energy}.
	\end{thm}
	In the same work \cite{JT}, it was also established that the family of solutions satisfying the hypotheses of Theorem \ref{S:GLOBAL0} is non-empty. Using ideas proposed in \cite{belinski1979stationary}, it was proved that there exists a family of solitons-like solutions for the PCF equation \eqref{S:eqPCF}. The construction and characterization of these solutions are briefly shown in the following section. 

\subsection{Soliton solutions}\label{1soliton}
As it was shown in \cite{belinskii1978integration} (see also \cite{belinski2001gravitational}), equation  \eqref{S:eqPCF} has $N$-soliton solutions. More precisely, in \cite{belinskii1978integration} it was proposed a transformation that constructs the so-called gravisolitons in the context of the vacuum Einstein equations in General Relativity, under certain conditions on coordinate symmetry. Using this transformation, Hadad \cite{yaronhadad_2013} explicitly showed the structure of the $N$-soliton for the PCF model \eqref{S:eqPCF}. The application of this transformation can be done since this model can be identified with the so-called \emph{reduced Einstein equation} given by
\begin{equation}\label{reducedEE}
		\partial_t\left(\alpha\partial_t g g^{-1}\right) -\partial_x\left( \alpha \partial_x g g^{-1} \right)=0, \quad (t,x)\in \mathbb R\times \mathbb R,
\end{equation}
with $\det g=\alpha$. Note that \eqref{reducedEE} with $\alpha=1$ reduces to \eqref{S:eqPCF}. 
It is important to note that although the identification of \eqref{S:eqPCF}  with  \eqref{reducedEE} can be made, this special class of solutions has no relevance for the gravitational field. However, even in the case ($\alpha\equiv 1$), the PCF model is sufficiently rich to produce complex dynamics and would formally have nontrivial solutions even when $\alpha$ is constant. In \cite{MT2023}, we considered the more demanding case when $\alpha$ is non-constant.

We will now briefly discuss the general structure of the 1-soliton of \eqref{S:eqPCF}, as well as the particular elements that characterize it. One starts using diagonal backgrounds, also called \emph{seed metric},  of the form 
\be\label{seedG}
g^{(0)} =(e^{\Lambda^{(0)}} ,e^{-\Lambda^{(0)}}),\quad
\partial_t^2\Lambda^{(0)}-\partial_x^2\Lambda^{(0)}=0.
\ee	
In this case, if we want to identify the solution in terms of the fields $\Lambda$ and $\phi$ in (\refeq{S:PCF}), we must set $\Lambda=\Lambda^{(0)}$, $\phi=n\pi,$ with $n\in \mathbb{Z}$, and $\alpha=1$. The gauge choice for us will be $n=0$. 	
\subsubsection{Singular solitons}	
Hadad \cite{yaronhadad_2013} described the 1-soliton solution, which is obtained by taking $\Lambda^{(0)}=t$ (time-like)  and $\phi^{(0)}=0$. Note that with this choice the energy in  \eqref{S:energy} is not well-defined. Indeed, the energy proposed in (\refeq{S:energy}) is not finite. Hadad computed the corresponding 1-soliton solution using Belinski-Zakharov techniques, obtaining  the corresponding solitons, as follows
	\begin{equation*}
		g^{(1)}=	\left[\begin{array}{cc}
			\dfrac{e^{t}Q_c(x-vt)}{Q_c(x-vt-x_0)} & -\dfrac{1}{c}Q_c(x-vt)\\
			-\dfrac{1}{c}Q_c(x-vt) & \dfrac{e^{-t}Q_c(x-vt)}{Q_c(x-vt+x_0)}
		\end{array} \right],
	\end{equation*}
	where, for a fixed parameter $\mu > 1$, one has 
	\[
	Q_c(\cdot)=\sqrt{c}\sech(\sqrt{c} (\cdot)), \quad c=\left(\dfrac{2\mu}{\mu^2-1}\right)^2, \quad v=-\dfrac{\mu^2+1}{2\mu}<-1, \quad \hbox{and}  \quad x_0=\dfrac{\ln |\mu|}{\sqrt{c}}.
	\]
	Notice that the first component of $g^{(1)}$ grows in time. Therefore, we have a superluminal traveling soliton which moves to the left  (if $\mu > 0$). In \cite{JT}, we proposed a modification of this ``degenerate'' soliton solution by cutting off the infinite energy part, profiting of the wave-like character of solutions $\Lambda^{(0)}$. Although it is not so clear that they are physically meaningful, these new solutions have finite energy and local well-posedness properties in a vicinity. The stability of this solution was not clear at the moment. 

\subsubsection{Finite energy, nonsingular solitons}
       Now, we consider smooth functions $\theta, \sigma \in  C^{2}_c(\mathbb R)$, and $0<\mu<1$. For any $\lambda>0$ and small $\varepsilon >0$ , let
\be\label{BackSeed}
\Lambda^{(0)}_{\varepsilon} (t,x) := \lambda + \varepsilon( \theta(x+t )+\sigma(x-t)), \quad \phi^{(0)}: =0.
\ee		
Clearly $\Lambda^{(0)}_{\varepsilon}$ {\bf solves the wave equation} in $1+1$  dimensions \eqref{seedG} and it has finite energy \eqref{S:energy}. This will be for us the background seed. Let us define the following functions:
\begin{equation}
\gamma_1  :=\lambda +\varepsilon \theta(x+t); \qquad  \gamma_2:= \varepsilon \sigma(x-t), \qquad  \gamma:= \gamma_1+\gamma_2.
\end{equation}
Now, the corresponding 1-soliton is given by (for full details see \cite[Section 5.2]{JT}):
\begin{equation}\label{1Soliton}
\begin{aligned}
	\hat{\Lambda}_{\varepsilon} \equiv B(t,x)= &~{} \arcosh\left(|v|\cosh(\gamma) -\frac{1}{\sqrt{c}}\tanh(\beta \gamma_1+\beta^{-1}\gamma_2)\sinh( \gamma) \right),\\
	\hat{\phi}_{\varepsilon} \equiv D(t,x)=  &~{}  \frac{\pi}{4}-\dfrac{1}{2}\arctan\left(\cosh(\beta \gamma_1+\beta^{-1}\gamma_2 )\cosh(\gamma) [\tanh(\beta \gamma_1+\beta^{-1}\gamma_{\color{black}2})+v\sqrt{c}\tanh(\gamma)  ]\right),
\end{aligned}	
\end{equation}
with $ \beta=\frac{\mu +1}{\mu -1}$. Similar to \cite{JT}, one has that $B$ is well-defined, $B(t,x) > 0 $ for all $t,x\in \mathbb{R}$, and for each $t$, $B$ is a bounded function. Again following \cite{JT}, one can write $B=  \tilde{\lambda} +\tilde B$, $\tilde{\lambda}:=\lim_{x\to \infty}B(t=0,x)>0$. Moreover, 
\[
	\tilde B |_{\{t=0\}}= \varepsilon\tilde B_0, \quad \mbox{with}\quad  \tilde B_0\in C_0^{2}(\mathbb{R}).
\]
where $\tilde B_0$ is bounded in $\varepsilon$. A similar computation can be done for $D$ and its time derivative, see \cite{JT}. Therefore, the setting of Theorem \ref{S:GLOBAL0} is satisfied.

In this paper, in order to get global solutions, we shall assume perturbed initial data  of the form
	\begin{equation}\label{initial_data}
		\begin{cases}
			(\Lambda, \phi)|_{\{t=0\}}= (B + \varepsilon z_0,D+ \varepsilon  s_0), & (z_0,s_0)\in C_c^{\infty}(\mathbb{R})^2, \\
			(\partial_t\Lambda,\partial_t \phi)|_{\{t=0\}}= (B_t + \varepsilon w_0, D_t + \varepsilon m_0), & (w_0,m_0) \in C_c^{\infty}( \mathbb{R})^2.  
		\end{cases}
	\end{equation} 
Then, by the previous non-degeneracy analysis and Theorem \ref{S:GLOBAL0}, there exists $\varepsilon_0$ such that if $\varepsilon < \varepsilon_0$, the unique solution remains smooth for all time and it has finite conserved energy \eqref{S:energy}. This does not ensure that the perturbations $[z,w,s,m](t)$ will remain small compared with $\tilde B$ and $\tilde D$ after a large time, but our goal will be to guarantee that if they are small at time zero, they will remain small at large time.

\subsection{Main results} Having described into detail the PCF 1-soliton solutions in the previous subsection, the purpose of this paper is to give a first proof of the fact that the 1-soliton \eqref{1Soliton} of the PCF model is orbitally stable under small perturbations, well-defined in the natural energy space associated to the problem. The stability study will be done by addressing equation \eqref{S:PCF} and using the description of the fields $B$ and $D$ given by \eqref{1Soliton}. Our main theorem is the following:

\begin{thm}\label{MT1}
There exists $\varepsilon_0>0$ such that if $0<\varepsilon<\varepsilon_0$, the following holds.
There exist $C ,\delta_0>0$ such that, if  $0<\delta<\delta_0$, and $[z_0,w_0,s_0,m_0]$ is given as in \eqref{initial_data}, then:

\begin{enumerate}
\item External energetic control. Assume that 
\[
\int \left( \frac{1}{2} (w_0^2+ z_{0,x}^2) +2\sinh^2(B+z_0) ( s_{0,x}^2+m_0^2) \right)(t,x)dx < \delta,
\]
\noindent
{\color{black}and that at time $t=0$ one has $\spp(\theta')\cup \spp(\sigma')\subseteq \{\xi \in \mathbb{R}: |\xi| < R \}$, for some $R>0$.}
Then
\begin{itemize}
\item For all time, one has a crossed global control:
\begin{equation}\label{diferencia}
 \int_{\mathbb R} \left(\frac12 (z_x-w)^2+ 2\sinh^2(B+z)(s_x-m)^2 \right) (t,x)dx < 3\delta.
\end{equation}
\item Inside light-cone convergence. For any $v\in (-1,1)$ and $\omega(t)= t/ \log^{2}t$, one has 
\begin{equation}\label{velocidad_menor1}
\lim_{t\to +\infty}\int_{vt -\omega(t)}^{vt+ \omega(t)} \left( w^{2}+z_x^{2}+\sinh^{2}(B+z)(m^{2}+s_x^{2}) \right)(t,x)dx=0.
\end{equation}
\item  Exterior stability: for all time $t\geq 0$ {\color{black} and $R$ as above},
\begin{equation}\label{exterior_stability}
\int_{  |x+t| \geq R}   \left( \frac{1}{2} (w^2+ z_x^2) +2\sinh^2(B+z) ( s_x^2+m^2) \right)(t,x) dx < \delta.
\end{equation}
\end{itemize}
\smallskip
\item Full orbital stability. Assume now $[z_0,w_0,s_0,m_0]\in C_c^{\infty}(\mathbb R)^4$ be initial data as in \eqref{initial_data} such that
\begin{equation}\label{ICs}
\sum_{k=0,1} \int(1+ |x|^2)^{1+\eta} \left( (\partial_x^k w_0 )^2+ (\partial_x^{k+1} z_0)^2 +(\partial_x^k m_0 )^2 +(\partial_x^{k+1} s_0 )^2\right)dx < \delta^2,
\end{equation}
for $0< \eta <\frac13$.
Then the corresponding global solution to \eqref{ecuacion_final} given as
\begin{equation}\label{deco}
(B+z, \partial_t B+ w, D + s, \partial_t D+ m) 
\end{equation}
satisfies (see \eqref{energies} for the precise definition of norms) 
\begin{equation}\label{Conclusion}
\sup_{t\geq 0} \left( \mathcal{E}(t)  +\overline{\mathcal{E}}(t) \right)  \leq {\color{black} C}\delta^2.
\end{equation}
\end{enumerate}
\end{thm}

\medskip
In the next lines, we are going to prove statements \eqref{diferencia}, \eqref{velocidad_menor1},
\eqref{exterior_stability}, 
\eqref{Conclusion}. First of all notice that \eqref{diferencia} is a global in time and space property satisfied by perturbations of solitons. However, having direct control over each component is an important open question that probably requires the introduction of better decay properties of the initial data, as is done in \eqref{ICs}. Additionally, condition \eqref{ICs} is part of a more general energy norm condition, described as
  \[
  \mathcal{E}[z_0,w_0]  +\overline{\mathcal{E}}[s_0,m_0]  < \delta^2,
  \]
(see \eqref{energies} for the definition of these norms). It turns out that these norms are key to translate the smallness information of the problem, in addition to the smallness of the 1-soliton solution, represented by the parameter $\varepsilon$, which enters when computing space and time derivatives of the solution.

One may think that standard energetic Lyapunov control in terms of energy and momentum is the key to prove orbital stability. Unfortunately, in the case of PCF solitons, this is not the case. Indeed, energy and momentum are useless because they do not control the central region around the soliton. The best example of this fact is the estimate \eqref{diferencia}. The proof of Theorem \ref{MT1} needs new ideas. In \cite{JT}, the author showed that perturbations of \emph{small} 1-solitons lead to a global solution, but the smallness of the perturbation could not be preserved. In this paper we avoid this problem by proposing a new idea obtained from the general case of Einstein's field equations under the Belinski-Zakharov formalism \cite{MT2023}. Solitons are perturbations of the nonlinear primordial equations, and their size will be controlled using well-defined weighted norms. It will be particularly important to notice that solitons are assumed sufficiently small, but perturbations are considered smaller.

\subsection*{Organization of this work} This work is organized as follows. In Section \ref{sec:2a} some important definitions, notation and previous work are stated. In Section \ref{sec:2} we will introduce energy and momentum for perturbations of the soliton solution, and compute useful virial identities valid for data only in the energy space. Finally, in Section \ref{sec:4} we prove the nonlinear stability of perturbations of nonsingular PCF solitons. 

\subsection*{Acknowledgments} Part of this work was done while the third author visited the Department of Mathematics of the University of C\'ordoba (Spain). She deeply thanks the professors for their hospitality during some research stays. {\color{black} We thank the referees for their valuable comments and insights that helped to improve a first version of this work.}

\section{Preliminaries}\label{sec:2a}

\subsection{Soliton profiles}
Recall the solitons $B$ and $D$ defined in \eqref{1Soliton}. We define their time and space derivatives as follows: 
 \begin{equation*}
 \begin{cases}
f_1\equiv f_1(t,x;\lambda,\varepsilon):= |v|\tanh(\gamma)-\frac{\beta}{\sqrt{c}}\sech^2(\beta\gamma_1+\beta^{-1}\gamma_2)\tanh(\gamma)-\frac{1}{\sqrt{c}}\tanh(\beta\gamma_1+\beta^{-1}\gamma_2),\\
f_2\equiv f_2(t,x;\lambda,\varepsilon):= |v|\tanh(\gamma)-\frac{\beta^{-1}}{\sqrt{c}}\sech^2(\beta\gamma_1+\beta^{-1}\gamma_2)\tanh(\gamma)-\frac{1}{\sqrt{c}}\tanh(\beta\gamma_1+\beta^{-1}\gamma_2),\\
f_3\equiv f_3(t,x;\lambda,\varepsilon):=\sech(\gamma)\sqrt{\left(|v|\cosh(\gamma)-\frac{1}{\sqrt{c}}\sinh(\gamma)\tanh(\beta\gamma_1+\beta^{-1}\gamma_2)\right)^2-1}; 
\end{cases}
 \end{equation*}
then
\begin{equation}\label{pt_B}
 \partial_t B =\dfrac{\varepsilon(\theta^{\prime}(x+t)f_1-\sigma'(x-t)f_2)}{f_3}, \qquad \partial_x B =\dfrac{\varepsilon(\theta^{\prime}(x+t)f_1+\sigma'(x-t)f_2)}{f_3}. 
  \end{equation} 
Clearly $ \partial_t B |_{\{t=0\}} \in C_0^{1}(\mathbb{R})$. For $D$, we define
\begin{equation*}
\begin{cases}
&g_1\equiv g_1(t,x;\lambda,\varepsilon):=\beta+v\sqrt{c}+(1+\beta v\sqrt{c})\tanh(\beta\gamma_1+\beta^{-1}\gamma_2)\tanh(\gamma),\\
&g_2\equiv  g_2(t,x; \lambda,\varepsilon):=\beta^{-1}+v\sqrt{c}+(1+\beta^{-1} v\sqrt{c})\tanh(\beta\gamma_1+\beta^{-1}\gamma_2)\tanh(\gamma),\\
 &g_3\equiv g_3(t,x;\lambda,\varepsilon):=2\sech(\beta\gamma_1+\beta^{-1}\gamma_2)\sech(\gamma) \\
 &\qquad \qquad  \times\left(\left(\cosh(\gamma)\sinh(\beta \gamma_1+\beta^{-1}\gamma_2)+v \sqrt{c}\sinh(\gamma)\cosh(\beta\gamma_1+\beta^{-1}\gamma_2)\right)^2+1\right),
\end{cases}
\end{equation*}
then
 \begin{equation}\label{pt_D}
 \begin{aligned}
	\partial_t D= \dfrac{-\varepsilon(\theta^{\prime}(t+x)g_1-\sigma'(x-t)g_2)}{g_3}, \qquad \partial_x D= \dfrac{-\varepsilon(\theta^{\prime}(t+x)g_1+\sigma'(x-t)g_2)}{g_3},
	\end{aligned}
\end{equation}
which is also a localized function. Notice that $\partial_t B,\partial_t D \in L^2(\mathbb{R})$, and the total energy \eqref{S:energy} is finite. 

\subsection{Previous bounds on global solutions} Let us recall some previous information in \cite{JT} about the global solution in our problem. Consider
\begin{equation*}
	u:=\dfrac{t-x}{2}, \quad \underline{u}:= \dfrac{t+x}{2},
\end{equation*}
and the two null vector fields, defined globally as 
\[
L=\partial_t + \partial_x,\quad \underline{L}= \partial_t - \partial_x.
\]
${\color{black}\Sigma}_{t_0}$  denotes the region 
\begin{align}\label{S_t}
	{\color{black}\Sigma}_{t_0} := \{(t,x):\; t=t_0 \}.
\end{align}
$D_{t_0}$ denotes the following region of spacetime 
\begin{equation*}
	D_{t_0}:= \{ (t,x):\;  0 \leq t \leq t_0 \}, \quad D_{t_0}=\bigcup_{0\leq t \leq t_0} {\color{black}\Sigma}_{t_0}.
\end{equation*}
The level sets of the functions $u$ and $\underline{u}$ define two global null foliations of $D_{t_0}$. More precisely, given $t_0>0$, $u_0$ and $\underline{u}_0$, we define the rightward null curve segment {\color{black}$C_{u_0}$} as:
\begin{equation}\label{C1}
	C_{u_0} := \left\{(t,x): \; u=\frac{t-x}{2} =u_0,\, 0\leq t \leq t_0\right\},
\end{equation}
and the segment of the null curve to the left {\color{black}$\underline{C}_{\underline{u}_0}$}as:
\begin{equation}\label{C2}
	\underline{C}_{\underline{u}_0} := \left\{(t,x): \; \underline{u}=\frac{t+x}{2} =\underline{u}_0,\, 0\leq t \leq t_0\right\}.
\end{equation}
The space time region $D_{t_0}$ is foliated by {\color{black}$\underline{C}_{\underline{u}_0}$} for $\underline{u}\in \mathbb{R}$, and by {\color{black}$C_{u_0}$} for $u\in \mathbb{R}$.

In the  same way as in \cite{Luli2018, JT} we consider the weight function $\varphi$ defined as 
\begin{equation}\label{varphi}
\varphi(u):=(1+|u|^2)^{1+\eta}\quad  \mbox{with} \quad 0< \eta<1/3.
\end{equation}
Finally, we will consider the following energy estimate proposed in \cite{alinhac2009hyperbolic, Luli2018} for the scalar linear wave equation $\square \psi = \rho.$ There exists $C_0>0$ such that 

\begin{equation}
	\begin{aligned}\label{EEnergia}
		& ~{}\int_{\Sigma_t} \left[ \varphi(u)\abs{\underline{L}\psi}^2+ \varphi(\underline u) \abs{L\psi}^2\right]dx \\
		& \qquad +\sup_{\underline u\in \mathbb{R}}\int_{\underline{C}_{\underline u}} \varphi(u)\abs{\underline{L}\psi}^2d\tau + \sup_{u\in \mathbb{R}}\int_{C_{u}} \varphi(\underline u)\abs{L\psi}^2 d\tau \\
		&~{} \qquad \leq C_0\int_{\Sigma_0}\left[ \varphi(u)\abs{\underline{L}\psi}^2+\varphi(\underline u)\abs{L\psi}^2 \right]dx + C_0 \iint_{D_t} \left[\varphi(u)\abs{\underline{L}\psi} + \varphi(\underline u)\abs{L\psi}\right]\abs{\rho} d\tau dx.
	\end{aligned}
\end{equation}
Based on this estimate, and in the setting of equation \eqref{ecuacion_final}, we define the space-time weighted energy norms valid for $k=0,1$:
\begin{equation}\label{energies_0}
\begin{aligned}
	&	\mathcal{E}_k[\tilde \Lambda](t)=\int_{{\color{black}\Sigma}_t} \left[ \varphi(u)\abs{\underline{L}\partial_x^k\tilde{\Lambda}}^2+  \varphi(\underline{u})\abs{L\partial_x^k \tilde{\Lambda}}^2\right]dx,\\
	&	\overline{\mathcal{E}}_k[ {\color{black}\phi}](t)=\int_{{\color{black}\Sigma}_t} \left[ \varphi(u) \abs{\underline{L}\partial_x^k\phi}^2+  \varphi(\underline{u})\abs{L\partial_x^k \phi}^2\right]dx,\\
	& \mathcal{F}_k[\tilde \Lambda](t)= \sup_{{\color{black} \underline u}\in \mathbb{R}} \int_{{\color{black} \underline{C}_{\underline u}}} \varphi(u) \left| \underline{L}\partial_x^k\tilde{\Lambda} \right|^2ds+ \sup_{{\color{black} u}\in \mathbb{R}} \int_{{\color{black} C_{u}}}\varphi(\underline{u}) \abs{L\partial_x^k \tilde{\Lambda}}^2ds,\\
	& \overline{\mathcal{F}}_k[ {\color{black}\phi}](t)= \sup_{{\color{black}\underline u}\in \mathbb{R}} \int_{C_{{\color{black} \underline u}}} \varphi(u) \left| \underline{L}\partial_x^k\phi \right|^2ds+ \sup_{{\color{black} u}\in \mathbb{R}} \int_{C_{{\color{black} u}}} \varphi(\underline{u})\abs{L\partial_x^k \phi}^2ds.
\end{aligned}
\end{equation}
Finally, we define the total energy norms as follows:
\[
\mathcal{E}(t)= \mathcal{E}_0 (t)+\mathcal{E}_1(t).
\] 
Analogously one defines $\mathcal{F}(t)$, $\overline{\mathcal{E}}(t)$, and $\overline{\mathcal{F}}(t).$  Then from \cite{JT} we know that there exists $C>0$ such that for all $t\geq 0$,
\[
\mathcal{E}(t)+\mathcal{F}(t) \leq C\varepsilon^2,	\qquad \overline{\mathcal{E}}(t)+ \overline{\mathcal{F}}(t) \leq C\varepsilon^2, 
\]
and ($\tilde{\lambda}:=\lim_{x\to \infty}B(t=0,x)>0$)
\begin{equation}\label{condicionlambda0}
	\sup_{t \geq 0} \norm{\tilde{\Lambda}}_{L^{\infty}(\mathbb{R})} \leq \dfrac{\tilde \lambda}{2}.
\end{equation}
From this last fact, we have the following corollary:

\begin{cor}\label{condicionlambda_z}
Let $z(t)$ be defined in \eqref{deco}. One has for all time $t\geq 0$,
\begin{equation}\label{cotas}
\begin{aligned}
 0<c_0(\tilde \lambda) \leq \sinh (B + z) \leq c_1(\tilde \lambda),\quad 0<d_0(\tilde \lambda) \leq \cosh (B + z) \leq d_1(\tilde \lambda).
\end{aligned}
\end{equation}
\end{cor}

\begin{proof}
Recall $B$ in \eqref{1Soliton}. Since $B+z = \tilde \lambda + \tilde \Lambda$, we first have $ \frac{\lambda}4 \leq \gamma=\lambda+\varepsilon (\theta(\xi)+\sigma(\xi))\leq \frac54\lambda$, and from \eqref{condicionlambda0},
\[
\frac12 \tilde \lambda \leq |B+z| \leq \frac32 \tilde \lambda. 
\]
The remaining bounds in \eqref{cotas} are direct consequences of the previous inequality.
\end{proof}

\subsection{Local decay} Recall the following result from \cite{JT}:

	\begin{thm}\label{LTD}
	Under a finite energy assumption, every globally defined solution to the PCF model satisfies the following convergence property: for any $v\in (-1,1)$ and $\omega(t)= t/ \log^{2}t$, one has 
	\[
	\lim_{t\to +\infty}\int_{vt -\omega(t)}^{vt+ \omega(t)}
		\left((\partial_{t}\Lambda)^{2}+(\partial_{x}\Lambda)^{2}+\sinh^{2}{\Lambda}((\partial_{t}\phi)^{2}+(\partial_{x}\phi)^{2}) \right)(t,x)dx=0.
	\]
	\end{thm}
	This result shows that finite-energy global PCF solutions must decay inside the light cone. Notice that this is a property satisfied by any globally-defined finite-energy solution. An immediate consequence of the previous result is the following decay property, which proves \eqref{velocidad_menor1}.

\begin{cor}\label{cor2p2}
For any $v\in (-1,1)$ and $\omega(t)= t/ \log^{2}t$, one has 
	\[
	\lim_{t\to +\infty}\int_{vt -\omega(t)}^{vt+ \omega(t)}
		\left( w^{2}+z_x^{2}+\sinh^{2}(B+z)(m^{2}+s_x^{2}) \right)(t,x)dx=0.
	\]
\end{cor}

\begin{proof}
An immediate consequence of the Theorem \ref{LTD} and the fact that $\partial_t B, \partial_x B$ and $\partial_t D, \partial_x D$ are supported in $|x+t|\leq R$ {\color{black}(note that they are also supported in $|x-t|\leq R$). Notice that this result suggest non existence of solitary waves in the considered integration region.}
\end{proof}

\section{Exterior stability bounds}\label{sec:2}

 Recall from \eqref{S:energy} and \eqref{S:momentum} the energy and momentum
\[
E[\Lambda, \Lambda_t,\phi, \phi_t](t)  = \int e(t,x)dx, \quad P[\Lambda, \Lambda_t,\phi, \phi_t](t)  = \int p(t,x)dx,
\]
where $e$ and $p$ are the energy and momentum densities:
 \begin{equation}\label{S:energy_density}
e: =\dfrac{1}{2}((\partial_{x}\Lambda)^{2}+(\partial_{t}\Lambda)^{2})+2\sinh^{2}(\Lambda)((\partial_{x}\phi)^{2}+(\partial_{t}\phi)^{2}),
	\end{equation}
\begin{equation}\label{S:momentum_density}
 p: = \partial_x\Lambda\partial_t\Lambda + 4\sinh^{2}(\Lambda)\partial_x\phi\partial_t\phi .
\end{equation} 
The  system \eqref{S:PCF} can be expanded around the perturbation of the soliton ${\bf B + u}$,  as follows: for $\mathbf{B}:= [B,\partial_t B,D,\partial_t D]$ and for any $\mathbf{u}=[z,w,s,m]$, we have that $\mathbf{u}$ formally satisfies ($\Box z =z_{tt}-z_{xx}$)
\begin{equation}
\begin{cases}\label{SistemaP_0}
\square z =&-2\sinh(2(B+z)) \left(2D_xs_x-2D_ts_t+s_x^2-s_t^2 \right)\\
&-2\left(2\sinh(2B)\sinh^2(z)+\sinh(2z)\cosh(2B)\right)\left(D_x^2-D_t^2 \right),\\
\square s=&-\dfrac{\sinh(2(B+z))}{\sinh^2(B+z)} \left( D_tz_t -D_xz_x+B_ts_t-B_xs_x+s_t z_t -s_x z_x \right)\\
& +\left(\frac{\sinh(2B)}{\sinh^2(B)}-\frac{\sinh(2(B+z))}{\sinh^2(B+z)} \right)\left( D_tB_t-D_xB_x \right).
\end{cases}
\end{equation}
This system can be written in matrix form as
\begin{equation}
\begin{cases}\label{SistemaP}
z_t=w,   & s_t=m,\\
w_t-z_{xx}=& -2\sinh(2(B+z)) \left(2D_xs_x-2D_tm+s_x^2-m^2 \right)\\
&-2\left(2\sinh(2B)\sinh^2(z)+\sinh(2z)\cosh(2B)\right)\left(D_x^2-D_t^2 \right),\\
m_t-s_{xx}=& -\dfrac{\sinh(2(B+z))}{\sinh^2(B+z)} \left( D_tw -D_xz_x+B_tm-B_xs_x+wm -s_x z_x \right)\\
& +\left(\frac{\sinh(2B)}{\sinh^2(B)}-\frac{\sinh(2(B+z))}{\sinh^2(B+z)} \right)\left( D_tB_t-D_xB_x \right).
\end{cases}
\end{equation}
Recalling  the conservation laws of the model  \eqref{S:energy} and \eqref{S:momentum}, and considering the expansion of the energy and momentum densities around the 1-soliton, we get

\begin{lem}
Let $(t,x)$ be such that $\mathbf{u}(t,x)$ is well-defined. Then a.e.
\begin{equation}
\begin{aligned}\label{expansionP}
2(e-p)[\mathbf{B}+\mathbf{u}]=&(B_x-B_t)^2+4\sinh^2(B+z)(D_x-D_t)^2\\
&+2(B_x-B_t)(z_x-w)+8\sinh^2(B+z)(D_x-D_t)(s_x-m)\\
 &+(z_x-w)^2+ 4\sinh^2(B+z)(s_x-m)^2 \geq 0,
\end{aligned}
\end{equation}
where densities $e$ and $p$ are given in \eqref{S:energy_density} and \eqref{S:momentum_density}. 
\end{lem}
\begin{proof}
First of all, notice that
\begin{equation}\label{expansion_e}
\begin{aligned}
e[\mathbf{B}+\mathbf{u}] = &~{}  \frac{1}{2}B_x^2+ B_x z_x + B_t w+ \frac{1}{2}B_t^2 + {\color{black}2}\sinh^2(B+z) \left( D_x^2+2D_xs_x+2D_tm +D_t^2\right) \\
&~{} + \frac{1}{2}(z_x^2+  w^2) + 2 \sinh^2(B+z) \left( s_x^2+m^2 \right).
\end{aligned}
\end{equation}
On the other hand,
\begin{equation}\label{expansion_p}
\begin{aligned}
p[\mathbf{B}+\mathbf{u}] =&B_tB_x+B_t z_x +B_x w    + 4 \sinh^2(B+z) \left( D_xD_t+D_xm+D_ts_x \right)\\
&+ z_xw +{\color{black}4}\sinh^2(B+z)\left( s_x m\right).
\end{aligned}
\end{equation}
Subtracting both identities we get \eqref{expansionP}:
\begin{equation*}
\begin{aligned}
2(e-p)[\mathbf{B}+\mathbf{u}]=&(B_x-B_t)^2+4\sinh^2{\color{black}(B+z)}(D_x-D_t)^2\\
&+2(B_x-B_t)(z_x-w)+8\sinh^2(B+z)(D_x-D_t)(s_x-m)\\
 &+(z_x-w)^2+ 4\sinh^2(B+z)(s_x-m)^2.
\end{aligned}
\end{equation*}
The proof is complete.
\end{proof}

Now we are ready to prove \eqref{diferencia}.
\begin{cor}
For all $t\geq 0$,
\[
\begin{aligned}
& \int_{\mathbb R} \left( (z_x-w)^2+ 4\sinh^2(B+z)(s_x-m)^2 \right) (t,x)dx  \\
& = \int_{\mathbb R} \left( (z_{0,x} -w_0)^2+ 4\sinh^2(B+z_0)(s_{0,x}-m_0)^2 \right) (x)dx<3\delta.
\end{aligned}
\]
\end{cor}
The previous result establishes that the differences $z_x-w$ and $s_x-m$ have good behavior in time, and remain bounded. An important part of the proof of Theorem \ref{MT1} will be to get better control on each part of the perturbation ${\bf u}$ separately.  Let us calculate the variation with respect to $t$ of the quantity \eqref{expansionP}, using the system \eqref{SistemaP}.
Let 
\begin{equation}\label{VRF}
\begin{aligned}
\hat e[\mathbf u](t):=&~{} \frac{1}{2} (w^2+ z_x^2) +2\sinh^2(B+z) ( s_x^2+m^2) ,\\
\hat p[\mathbf u](t):=&~{}z_x w +4\sinh^2(B+z) s_xm ,\\
F_p[ \mathbf u](t):=&~{} z_x(B_{tt}-B_{xx})+4\sinh^2(B+z)s_x(D_{xx}-D_{tt})\\
&~{}+2\sinh(2(B+z))(z_x(D_t^2-D_x^2+2s_tD_t)+2s_x(B_xD_x-B_tD_x-B_tD_t-z_tD_t))\\
&~{} +2\sinh(2(B+z))(s_x^2-s_t^2)B_x,\\
F_e[ \mathbf u](t):=&~{} z_t(B_{xx}-B_{tt})+4\sinh^2(B+z)s_t(D_{xx}-D_{tt})\\
&~{}+2\sinh(2(B+z)) \left(s_t(2B_xD_x-2D_tB_t+2z_xD_x)+z_t(D_t^2-D_x^2-2s_xD_x)\right) . \end{aligned}
\end{equation}
In this case, $\hat e$ and $\hat p$ are localized versions of $e$ and $p$ \eqref{S:energy_density}-\eqref{S:momentum_density}. Notice that $|\hat p[\mathbf u](t)| \leq |\hat e[\mathbf u](t)|$. Then, we claim the following
\begin{lem}\label{VR}  
For all $t\geq 0$, it holds that
\begin{equation}\label{2nd_order}
\partial_t \hat e[\mathbf u](t)= \partial_x\hat p[\mathbf u](t) +F_e [\mathbf u](t), \quad \partial_t \hat p[\mathbf u](t)= \partial_x\hat e[\mathbf u](t) +F_p [\mathbf u](t).
\end{equation}
\end{lem}
\begin{proof} 
From \cite[Lemma 4.2]{JT}, we know that $\partial_t e = \partial_x p$ and $\partial_t p = \partial_x e$, with $e$ and $p$ given in \eqref{S:energy_density}-\eqref{S:momentum_density}.
We have from \eqref{expansion_e}-\eqref{expansion_p},
\[
\begin{aligned}
 \partial_t \hat e[\mathbf u](t) = &~{}  -\partial_t \left(\frac12 (B_x^2{\color{black}+}B_t^2)+B_xz_x+B_tw+{\color{black}2}\sinh^2(B+z)(D_x^2+2D_xs_x+D_t^2+2D_tm) \right) \\
 &~{} + \partial_x\left( B_xB_t+B_xw+z_xB_t+4\sinh^2(B+z)(D_xD_t+D_xm+s_xD_t) + \hat p[\mathbf u] \right).
\end{aligned}
\]
Using \eqref{SistemaP_0},
\[
\begin{aligned}
&\partial_t \left(\frac12 (B_x^2{\color{black}+}B_t^2)+B_xz_x+B_tw+{\color{black}2}\sinh^2(B+z)(D_x^2+2D_xs_x+D_t^2+2D_tm) \right) \\
& = ~{}B_xB_{xt}+B_tB_{tt}+z_xB_{xt}+z_{xt}B_x+z_tB_{tt}\\
&\quad ~{} +B_t\left(z_{xx}+B_{xx}-B_{tt}-2\sinh(2(B+z))(D_x^2+2D_xs_x+s_x^2-D_t^2-s_tD_t-s_t^2) \right)\\
& \quad ~{}+ 4\sinh^2(B+z)D_t(s_{xx}+D_{xx}-D_{tt})\\
& \quad ~{}  -4\sinh(2(B+z))(D_tB_t+D_tz_t+s_tB_t+s_tz_t-D_xB_x-D_xz_x-s_xB_x-s_xz_x). 
\end{aligned}
\]
Simplifying,
\[
\begin{aligned}
&\partial_x \left( B_xB_t+B_xw+z_xB_t+4\sinh^2(B+z)(D_xD_t+D_xm+s_xD_t)  \right) \\
& = ~{} \partial_x \left( B_xB_t+B_xw+z_xB_t \right) \\
& \quad ~{} +4    \sinh(2(B+z))(B_x+z_x) \left( D_xD_t+D_xm+s_xD_t\right)\\
& \quad ~{} +4\sinh^2(B+z)(D_xxD_t+D_xD_{tx}+D_xxs_t+D_xs_{tx}+s_{xx}D_t+s_xD_{tx}).
\end{aligned}
\]
We conclude now that
\[
\begin{aligned}
 \frac{d}{dt} \hat e[\mathbf u](t) = &~{} \partial_x \hat p[\mathbf u] (t) + F_e[\mathbf u](t).
\end{aligned}
\]
This is the first identity in \eqref{2nd_order}. On the other hand, using  \eqref{expansion_e}-\eqref{expansion_p} again,
\begin{equation*}
\begin{aligned}
\partial_t \hat p[\mathbf u]=& -\partial_t\left(B_xB_t+B_xz_t+B_tz_x+4\sinh^2(B+z)(D_xD_t+D_xm+s_xD_t) \right)\\
&+ \partial_x\left(\frac12(B_x^2+B_t^2)+B_xz_x+B_tw+{\color{black}2}\sinh^2(B+z)(D_x^2+D_t^2+2D_xs_x+2D_tm) \right)\\
& +\partial_x \hat e[\mathbf u].
\end{aligned}
\end{equation*}
Rearranging appropriately, using \eqref{SistemaP_0} and simplifying, finally we obtain the second identity in \eqref{2nd_order}
\begin{equation*}
\begin{aligned}
& \frac{d}{dt}  \hat p[\mathbf u] (t)  =  \partial_x \hat e[\mathbf u] (t)  +F_p[\mathbf u].
\end{aligned}
\end{equation*}
\end{proof}
The Lemma \ref{VR}, allows us to propose the following integral estimation, which will be the focus of the subsequent section:
\begin{lem}\label{VR2}
Let $\chi_1,\chi_2$  smooth bounded weight functions and $r\in \mathbb R$. Then
\begin{equation}\label{VR2_1}
\begin{aligned}
\frac{d}{dt}\int \hat e[\mathbf u](t) \chi_1(x + rt )  \chi_2(x-rt)dx =  & ~{} r \int \hat e[\mathbf u](t)(\chi_1'(x+rt)\chi_2(x-rt)-\chi_1(x+t)\chi'_2(x-rt))dx \\
 &~{}-\int \hat p [\mathbf u](t)(\chi_1'(x+rt)\chi_2(x-rt)+\chi_1(x+t)\chi'_2(x-rt))dx\\
&~{} + \int F_e[\mathbf u](t) \chi_1(x+rt)\chi_2(x-rt) dx,\\
\frac{d}{dt}\int \hat p[\mathbf u](t)\chi_1(x+ rt )\chi_2(x-rt)dx  =&  ~{} r \int \hat p[\mathbf u](t)(\chi_1'(x+rt)\chi_2(x-rt){\color{black}-}\chi_1(x+t)\chi'_2(x-rt))dx\\
&~{}-\int \hat e[\mathbf u](t)(\chi_1'(x+rt)\chi_2(x-rt){\color{black}+}\chi_1(x+t)\chi'_2(x-rt))dx\\
&~{}  + \int F_p[\mathbf u](t) \chi_1(x+rt)\chi_2(x-rt) dx.
\end{aligned}
\end{equation}
\end{lem}
\begin{proof}The proof of the Lemma \ref{VR2} is direct from the Lemma \ref{VR} and  integration by parts. 
\end{proof} 

Note that from \eqref{pt_B}-\eqref{pt_D}, one has that $\partial_tB, \partial_t D$ are compactly supported functions. Assume that at time $t=0$ one has $\spp(\theta')\cup \spp(\sigma')\subseteq \{\xi \in \mathbb{R}: |\xi| < R \}$, for some $R>0$.

\begin{lem}[Exterior stability]
Assume that ${\bf B} + {\bf u}$ is globally defined. Then for any $0<\delta <R$ one has
\[
\int_{  |x+t| \geq R} \hat e[\mathbf u](t) dx \leq \int_{ |x| \geq R} \hat e[\mathbf u](0) dx.
\]
\end{lem}
This bound proves \eqref{exterior_stability}.

\begin{proof}
Let $\chi_1, \chi_2$ be smooth cut-off functions such that 
\[
0\leq \chi_1 \leq 1, \quad \chi'_1 \leq 0, \quad \chi_1(s)=0, \; s\geq -R , \quad \chi_1(s)=1, \; s\leq -R-1,
\] 
and
\[ 0\leq \chi_2 \leq 1, \quad \chi'_2 \geq 0, \quad \chi_2(s)=0, \; s\leq R , \quad \chi_2(s)=1, \; s\geq R+1.
\]
Now, if $r=1$, notice that $\chi_1 (x+t)$ has support in $x+t \leq -R$, and $\chi_2(x-t)$ has support in $x-t \geq R$. Since $|\hat p| \leq \hat e$, we have from \eqref{VR2_1},
\[
\begin{aligned}
\frac{d}{dt}\int \hat e[\mathbf u](t)\chi_1(x + t)\chi_2(x-t)dx=&\int\hat p |\chi_1'(x+t) |\chi_2(x-t)dx-\int \hat p \chi_1(x+t)\chi'_2(x-t)dx \\
& -\int\hat e |\chi_1'(x+t) |\chi_2(x-t)dx-\int \hat e \chi_1(x+t)\chi'_2(x-t)dx\\
&+\int F_e[\mathbf u](t) \chi_1(x+t)\chi_2(x-t) dx\\
\leq & \int F_e[\mathbf u](t) \chi_1(x+t)\chi_2(x-t) dx =0,
\end{aligned}
\]
proving the stability estimate for the left side. The other side is proved similarly. 
\end{proof}

\section{Interior control}\label{sec:4}

In this section we prove the last estimate in \eqref{Conclusion}. First of all, notice that one can write \eqref{SistemaP_0} as follows 
\begin{equation}
\begin{cases}\label{SistemaP_0_Q}
\Box z= 2\sinh(2B)Q_0(D,D)-2\sinh(2(B+z)) \left(Q_0(D,D)+ 2 Q_0(D,s) +Q_0(s,s) \right),\\
\Box s= -\frac{\sinh(2B)}{\sinh^2(B)}Q_0(D,B)+\frac{\sinh(2(B+z))}{\sinh^2(B+z)} \left(Q_0(D,B)+  Q_0(D,z)  + Q_0(B,s)  +Q_0(z,s) \right),
\end{cases}
\end{equation}
where $Q_0$ is the fundamental null form 
	\begin{equation*}
		Q_0 (\phi_1,\phi_2)= m^{\alpha \beta}\partial_{\alpha} \phi_1 \partial_{\beta} \phi_2,
	\end{equation*}
and where $m_{\alpha \beta}$ to denote the standard Minkowski metric on $\mathbb{R}^{1+1}$ with signature $(-1,1)$. It can also be noticed that the null structure is ``quasi-preserved'' after differentiating with respect to $x$, in the sense that 
	\begin{equation}\label{derivadanullC}
		\partial_xQ_0(\phi,\tilde{\Lambda})=Q_0(\partial_x \phi,\tilde{\Lambda})+Q_0(\phi,\partial_x \tilde{\Lambda}).
	\end{equation}
	Additionally, we have the following relation between  the null form and the Killing vector fields $L$ and $\underline{L}$
	\begin{equation}\label{importante}
		Q_0(\partial_{x}^p \phi, \partial_{x}^q \tilde{\Lambda}) \lesssim \abs{L\partial_x^p \phi}\abs{\underline{L}\partial_x^q \tilde{\Lambda}}+  \abs{\underline{L}\partial_x^p \phi}\abs{L\partial_x^q \tilde{\Lambda}},
	\end{equation}
where the implicit constant is independent of $(\tilde \Lambda, \phi)$. It is easy to check that (see \eqref{varphi}):
\begin{itemize}
	\item[(i)] Since $\varphi(u):=(1+|u|^2)^{1+\eta}$ and $0<\eta <1/3$, using the Bernoulli's inequality, we get that\\ $\varphi^{3/4}(\cdot)\leq (1+|2(\cdot)|)^2$.
	Thus,  since $\theta, \sigma \in C_c^{2}(\mathbb{R})$, one has that for some fixed constant $K_1,K_2>0$,
	\begin{equation}\label{eqn:alpha0}
		 | \theta^{(n+1)}(2 \underline u)| \leq \frac{K_1}{\varphi^{3/4}(\underline u)}, \quad n=0,1,
	\end{equation}
	and 
	\begin{equation}\label{eqn:alpha1}
		|\sigma^{(n+1)}(-2 u)| \leq \frac{K_2}{\varphi^{3/4}(u)}, \quad  n=0,1.
	\end{equation}

	\item[(ii)]  The following relations for the null vector field $L$ and $\underline{L}$ hold: 
	\begin{equation}\label{Llogalpha}
		\begin{aligned}
			&~{} |LD| = 2\varepsilon\left|\frac{g_1}{g_3}\right||\theta'(2\underline u)| \lesssim   \frac{ \varepsilon K_1}{\varphi^{3/4}(\underline u)},    \quad |\underline L D | =2\varepsilon\left|\frac{g_2}{g_3}\right||\sigma'(-2u)|\lesssim  \frac{\varepsilon K_2}{\varphi^{3/4}(u)},  \\
			&~{}  |LD_x|\lesssim \frac{\varepsilon K_1}{\varphi^{3/4}(\underline u)},  \qquad \qquad \qquad \quad  |\underline LD_x|\lesssim \frac{ \varepsilon K_2}{\varphi^{3/4}(u)}, 
		\end{aligned}
	\end{equation}
	and similar estimates for $B$.
\end{itemize}
Recall $\Sigma_t$, $C_{u}$ and $\underline{C}_u$ as introduced in \eqref{S_t}, \eqref{C1} and \eqref{C2}, respectively.  Recall the energies introduced in \eqref{energies_0}. We will adapt the proof of Theorem 3.1 in \cite{MT2023} to the case of solitons, to obtain that the following energies remain small: for $k=0,1$:
\begin{equation}\label{energies}
\begin{aligned}
	&	\mathcal{E}_k[z,w](t)=\int_{{\color{black}\Sigma}_t} \left[ \varphi(u)\abs{\underline{L}\partial_x^kz}^2+  \varphi(\underline u)\abs{L\partial_x^k z}^2\right]dx,\\
	&	\overline{\mathcal{E}}_k[s,m](t)=\int_{{\color{black}\Sigma}_t} \left[\varphi(u) \abs{\underline{L}\partial_x^ks}^2+  \varphi(\underline u)\abs{L\partial_x^k s}^2\right]dx,\\
	& \mathcal{F}_k[z,w](t)= \sup_{{\color{black} \underline u}\in \mathbb{R}} \int_{{\color{black} \underline{C}_{\underline u}}} \varphi(u) \left| \underline{L}\partial_x^k z \right|^2 + \sup_{{\color{black} u}\in \mathbb{R}} \int_{{\color{black} C_{u}}}  \varphi(\underline u) \abs{L\partial_x^k z}^2 ,\\
	& \overline{\mathcal{F}}_k[s,m](t)= \sup_{{\color{black}\underline u}\in \mathbb{R}} \int_{C_{{\color{black} \underline u}}} \varphi(u) \left| \underline{L}\partial_x^ks \right|^2 + \sup_{{\color{black} u}\in \mathbb{R}} \int_{C_{{\color{black} u}}}   \varphi(\underline u) \abs{L\partial_x^k s}^2 .
\end{aligned}
\end{equation}
Then, using \eqref{energies} we define the total energy norms as follows:
\[
\mathcal{E}(t) := \mathcal{E}_0[z,w] (t)+\mathcal{E}_1[z,w](t).
\] 
Analogously one defines $\mathcal{F}(t)$, $\overline{\mathcal{E}}(t)$, and $\overline{\mathcal{F}}(t).$  From \eqref{ICs} we know that $\mathcal{E}(0)  +\overline{\mathcal{F}}(0)  <C_1\delta^2$. The stability of the soliton perturbation \eqref{Conclusion} will be consequence of

\begin{thm}\label{thm_aux}
Assume that there exists $T^*>0$ such that, for all $t\in [0,T^{*}]$, the estimates
	\begin{align}\label{supuesto}
		&	\mathcal{E}(t)+\mathcal{F}(t) \leq {\color{black} 6C_1}C_0\delta^2,\\
		&	\overline{\mathcal{E}}(t)+ \overline{\mathcal{F}}(t) \leq {\color{black} 6C_1}C_0\delta^2\label{supuesto1},
	\end{align}
hold	for some $C_0>0$ and
	\begin{equation}\label{condicionlambda}
		\sup_{t \in [0,T^{*}]}\norm{z}_{L^{\infty}(\mathbb{R})} \leq \dfrac{\lambda}{2}.
	\end{equation}
	Then for all $t\in [0,T^*]$ there exists a universal constant $\delta_0$ (independent of $T^{*}$) such that the previous estimates are improved for all $\delta \leq \delta_0$. 
\end{thm}

\subsection{Proof of Theorem \ref{thm_aux}} 
As usual, we work with the system for $z$ in \eqref{SistemaP_0_Q}, the one for $s$ being very similar in both nature and estimates. Deriving \eqref{SistemaP_0_Q} and using  \eqref{derivadanullC}  we obtain:
\begin{equation}\label{derivadas}
	\square \partial_x z  = \rho_1 +\rho_2,
\end{equation}
where
\begin{equation}\label{rho_12}
	\begin{cases*}
		\rho_1:=4\sinh(2B)Q_0(\partial_x D,D)+4\cosh(2B)\partial_xBQ_0(D,D)\\
		\qquad  -4 \sinh(2(B +z)) \left( Q_0(\partial_xD,D) +Q_0(\partial_xD, s) + Q_0(D, \partial_x s ) + Q_0(s, \partial_x s) \right),\\
		\rho_2:= -4 \cosh(2(B+z)) (\partial_xB +\partial_x z) \left(Q_0(D,D)+ 2 Q_0( D, s)+Q_0( s, s)\right).  
	\end{cases*}
\end{equation}
Under the assumptions (\refeq{supuesto})-(\refeq{supuesto1})-(\refeq{condicionlambda}) for all $t\in [0,T^{*}]$, we assume that the solution remains regular, to later show that these bounds are maintained, with a better constant.  \\
Using \eqref{EEnergia}, with $\psi=\partial^k_x z$ and  \eqref{derivadas}-\eqref{rho_12}. Taking the sum over $k=0,1$, and using \eqref{condicionlambda}, we obtain
\begin{equation}\label{EE2}
	\begin{aligned}
		&\mathcal{E}(t)+\mathcal{F}(t) \leq 2C_0 \mathcal{E}(0)\\
		& +2 C_{\lambda}C_0\iint_{D_t} \left(  \varphi (u) | \underline{L}z|+ \varphi(\underline u)|L z|\right)  \left| Q_0(D,D)\right| \\
		&     + C_{\lambda} C_0  \iint_{D_t} \left(  \varphi (u) | \underline{L}z|+ \varphi(\underline u)|L z|\right)  \left| Q_0(D,s)\right|    + C_{\lambda}C_0  \iint_{D_t} \left( \varphi (u) | \underline{L}z|+ \varphi(\underline u) |L z|\right)   |Q_0(s,s)| \\
		&         +2C_{\lambda}C_0\iint_{D_t} \left( \varphi (u) | \underline{L}\partial_xz|+ \varphi(\underline u)  | L\partial_x z|\right) |\rho_1|  +2C_{\lambda}C_0 \iint_{D_t} \left( \varphi (u) | \underline{L}\partial_xz|+\varphi(\underline u)  | L\partial_x z|\right) |\rho_2|,	\end{aligned}
\end{equation}
where $C_{\lambda}:= C(c_0(\lambda),c_1(\lambda))$, same as in  \eqref{condicionlambda_z}. Recall the following result (\cite{Luli2018}, Lemma 3.2):
\begin{lem}\label{lema1}
	Under assumptions  \eqref{supuesto} and \eqref{supuesto1}, there exists $C_2>0$ such that:
	\begin{align*}
		&	|Lz(t,x)|\leq \dfrac{C_2 \delta}{(1+|\underline u |^2)^{1/2+\eta/2}}, &  	|Ls(t,x)|\leq \dfrac{C_2 \delta}{(1+|\underline u |^2)^{1/2+\eta/2}},\\
		& 	|\underline{L}z(t,x)|\leq \dfrac{C_2 \delta}{(1+|u |^2)^{1/2+\eta/2}}, &  	|\underline{L}s(t,x)|\leq \dfrac{C_2 \delta}{(1+|u |^2)^{1/2+\eta/2}}.
	\end{align*} 	
\end{lem} 
Note that  we can conveniently write $\partial_x z= \frac{1}{2}\left(Lz -\underline L z \right)$ and $\partial_xB = \frac{1}{2}\left( L B-\underline L B \right)$. Then, using \eqref{importante} we can estimate \eqref{EE2} as follows:
\begin{equation}\label{EE2b}
	\begin{aligned}
		&\mathcal{E}(t)+\mathcal{F}(t) \lesssim  2C_0 \mathcal{E}(0)+2C_0\sum_{n=1}^{6}\mathcal{T}_n.
\end{aligned}
\end{equation}	
where the integrals $\mathcal{T}_n$ are given by 
\[
\begin{aligned}	
\mathcal{T}_1:= &~{} \iint_{D_t} \left( \varphi (u) | \underline{L}z|+ \varphi(\underline u) |L z|\right)( |Ls||\underline L s| + (|Lz|+|\underline L z|)(|Ls||\underline L s|))\\
& \qquad+\iint_{D_t} \left( \varphi (u) | \underline{L}\partial_xz|+ \varphi(\underline u)  | L\partial_x z|\right)	(|\underline L \partial_x s ||Ls|+|\underline L s ||L\partial_x s|),
\end{aligned}
\]
\[
\mathcal{T}_2:=  \iint_{D_t} \left(  \varphi (u) | \underline{L}z|
		+ \varphi(\underline u)|L z|\right)|LD||\underline L D|+ \iint_{D_t} \left( \varphi (u) | \underline{L}\partial_xz|+ \varphi(\underline u)  | L\partial_x z|\right) (|L\partial_x D||\underline L D|+|\underline L \partial_x D||L D|),
\]
\begin{equation}\label{Tn}
\begin{aligned}	
		&\mathcal{T}_3:=  \iint_{D_t} \left(  \varphi (u) | \underline{L}z|+ \varphi(\underline u)|L z|\right)\left( \left| L s\right||\underline LD| +\left| \underline L s\right||LD|\right)\\	
			& \qquad  + \iint_{D_t} \left( \varphi (u) | \underline{L}\partial_xz|+ \varphi(\underline u)  | L\partial_x z|\right)	\left( |\underline L \partial_x s ||LD|+ | L\partial_x z|+|L\partial_x D||\underline L s|+ |\underline L \partial_x D|| L s|+| L \partial_x s ||\underline LD|\right),\\
			& \mathcal{T}_4:= \iint_{D_t} \left(  \varphi (u) | \underline{L}z|+ \varphi(\underline u)|L z|\right)|LB-\underline L B||LD||\underline L D|+ \iint_{D_t} \left( \varphi (u) | \underline{L}\partial_xz|+ \varphi(\underline u)  | L\partial_x z|\right)|LB||\underline L D|||LD|,\\		
& \mathcal{T}_5:= \iint_{D_t} \left( \varphi (u) | \underline{L}\partial_xz|+ \varphi(\underline u)  | L\partial_x z|\right)|Lz||\underline L D|||LD|+ \iint_{D_t} \left( \varphi (u) | \underline{L}z|+ \varphi(\underline u) |L z|\right)|Ls||LB||\underline L D|\\
& \qquad + \iint_{D_t} \left( \varphi (u) | \underline{L} z|+ \varphi(\underline u)  | Lz| \right)\left(|\underline L s||L B||\underline L D|+|\underline L s||\underline L B||LD|+|\underline L z||L D||\underline L D|+|Ls||\underline L B||\underline L D|\right),\\		
& \mathcal{T}_{6}:=\iint_{D_t} \left( \varphi (u) | \underline{L} z|+ \varphi(\underline u)  | Lz| \right)\left(|L z||\underline Ls||L D|+|L s||\underline Ls||L B|+|Lz||L s||\underline L D|+|\underline L z ||\underline L s| |L D|\right)\\
&\qquad + \iint_{D_t} \left( \varphi (u) | \underline{L} z|+ \varphi(\underline u)\right) \left(|\underline L z||L s||\underline L D|+ |Ls||\underline L s||\underline L B| \right).
\end{aligned}		
 \end{equation}
Essentially, we have 28 terms to control, but many of them are equivalent. Indeed, it will be sufficient to bound the terms corresponding to $\varphi (u)\underline L z$ and $\varphi(u)\underline L \partial_x z$, since by symmetry, the procedure for the other terms will be analogous.
 
 First, note that the integral $\mathcal{T}_1$ can be written as  $\mathcal{T}_1:= \sum_{j=1}^{5}\mathcal{T}_{1,j} $,  
 taking into account the Lemma \ref{lema1} and the estimates already established in \cite{JT,MT2023}, we have that $\mathcal{T}_{1,j}\lesssim \delta^3$ for $j\in \{1,4,5\}$, while $\mathcal{T}_{1,j}\lesssim \delta^4$ for $\{2,3\}.$ See these references for full details. 
 
 As follows, we will focus on the key terms in this framework, that is, the terms $\mathcal{T}_i$,  for $i\in \left\{2,3,...,6\right\}. $  From now on we define $K:= \max \{K_1,K_2\}$, where the constant $K_1,K_2$ are given by the estimate \eqref{Llogalpha}.
  Let us start with $\mathcal{T}_2$, taking $m,k\in \{0,1\}$ we have that all the integrals that define this term can be written as:
 \begin{equation}
  \iint_{D_t}\left(  \varphi (u) | \underline{L}\partial_x^kz|
		+ \varphi(\underline u)|L\partial_x^k z|\right)|L\partial_x^kD||\underline L\partial_x^m D|.
 \end{equation}
 We bound this term in the following form: using the estimates \eqref{Llogalpha}, \eqref{supuesto} and  H\"{o}lder inequality we get
 \begin{align*}
\iint_{D_t} & \left(  \varphi (u) | \underline{L}\partial_x^kz|
		+  \varphi(\underline u)|L\partial_x^k z|\right) |L\partial_x^kD||\underline L\partial_x^m D|\\
		& \lesssim \left( \iint_{D_t}\frac{ \varphi(u)|\underline L \partial_x^k z|^2}{\varphi(\underline u)}  \right)^{1/2}\left( \iint_{D_t} \frac{\varepsilon^4K_1^2K_2^2}{\varphi^{1/2}(u)\varphi^{1/2}(\underline u)}\right)^{1/2}\\
		&\qquad +\left( \int_{\mathbb{R}}\frac{1}{\varphi(\underline u)}\left[ \int_{C_{\underline u}} \varphi( u)|L z|^2ds\right]d\underline u\right)^{1/2} \left( \iint_{D_t} \frac{\varepsilon^4K_1^2K_2^2}{\varphi^{1/2}(u)\varphi^{1/2}(\underline u)}\right)^{1/2}     \lesssim \varepsilon^2K_1K_2\delta.
  \end{align*}
Thus $\mathcal{T}_2\lesssim  \varepsilon^2K^2\delta $.  The next term to  study is $\mathcal{T}_3$, in this case, taking  $m,k\in \{0,1\}$ and $\psi\in \{ D,s\},$ we get that all the integrals that define this term can be written as:
 \begin{equation}
  \iint_{D_t}\left(  \varphi (u) | \underline{L}\partial_x^kz|
		+ \varphi(\underline u)|L\partial_x^k z|\right)|L\partial_x^k\psi||\underline L\partial_x^m \psi|,
 \end{equation}
 using again \eqref{Llogalpha}, Lemma \ref{lema1} and  H\"{o}lder inequality we get that $\mathcal{T}_3\lesssim \varepsilon K \delta^2.$
 Now, for the term $\mathcal{T}_4$, if we take $\psi_1,\psi_2\in \{D,B\},$ all the integrals that define this term can be written as:
 \begin{equation}\label{T4}
  \iint_{D_t}\left(  \varphi (u) | \underline{L}\partial_x^kz|
		+ \varphi(\underline u)|L\partial_x^k z|\right)\left(|L\psi_1||\underline L\psi_2|| L \psi_2|+ |L\psi_1||\underline L\psi_2|| \underline L \psi_2|\right).
 \end{equation}
 Now, we can bound this term in the following form: using the estimates \eqref{Llogalpha}, \eqref{supuesto} and H\"{o}lder inequality  again, we get
 \begin{equation*}
 \begin{aligned}
  \iint_{D_t} \varphi (u) | \underline{L}\partial_x^kz|\left(|L\psi_1||\underline L\psi_2|| L \psi_2|\right)&\lesssim  \left(\iint_{D_t} \varphi(u)|\underline L \partial_x^kz|^2||L \psi_1 |^2 \right)^{1/2}\left(\iint_{D_t} \varphi(u)||L \psi_2 |^2|\underline L \psi_2 |^2 \right)^{1/2}\\
		& \lesssim\left(\iint_{D_t} \frac{\varepsilon^2K_1^2\varphi(u)|\underline L \partial_x^kz|^2|}{\varphi^{3/2}(\underline u)} \right)^{1/2}\left(\iint_{D_t}\frac{\varepsilon^4K_1^2K_2^2}{\varphi^{3/2}(\underline u)\varphi(u)} \right)^{1/2} \lesssim \varepsilon^3K^3\delta,
 \end{aligned}
 \end{equation*}
 and the estimate for the remaining terms in \eqref{T4} can be obtained in an analogous way, thus $\mathcal{T}_4\lesssim \varepsilon^3K^3\delta.$ The following term to study is $\mathcal{T}_5$, for this case,  we take $\psi\in \{s,z\}$  and again  $\psi_1,\psi_2\in \{ D,B\}.$   Then the terms that describe the integrals can be written, for $i,j=1,2,$ as follows:
\begin{equation}\label{T5}
  \iint_{D_t}\left(  \varphi (u) | \underline{L}\partial_x^kz|
		+ \varphi(\underline u)|L\partial_x^k z|\right)\left(|\underline L\psi||\underline L\psi_i|| L \psi_j|+| L\psi||\underline L\psi_i|| L \psi_j|\right),
 \end{equation}
 using \eqref{supuesto}, \eqref{Llogalpha} and  H\"{o}lder inequality  again, we get
 \begin{equation*}
 \begin{aligned}
   \iint_{D_t}  \varphi (u) | \underline{L}\partial_x^kz| \left(| L\psi||\underline L\psi_i|| L \psi_j|\right)\leq & \left(\iint_{D_t} \varphi(u) |\underline L \partial_x^k z|^2 |L\psi_j|^2 \right)^{1/2} \left(\iint_{D_t} \varphi(u) |\underline L \psi|^2 |\underline L\psi_i|^2 \right)^{1/2}\\
 \leq &  \left(\iint_{D_t} \frac{K_2^2 \varepsilon^2}{\varphi^{3/2}(\underline u)} \varphi(u) |\underline L \partial_x^k z|^2 \right)^{1/2} \left(\iint_{D_t} \frac{K_1^2 \varepsilon^2 }{\varphi^{3/2}(\underline u)} \varphi(u)|\underline L \psi|^2 \right)^{1/2} \\
  \lesssim &  K_1K_2\varepsilon^2 \delta^2.
 \end{aligned}
 \end{equation*}
For the last term $\mathcal{T}_6$, if take $\psi_1,\psi_2\in \{z,s\}$ and $\psi\in \{D,B\},$ we get that the integral terms, for $i,j=1,2$, can be rewritten as 
\begin{equation*}
\begin{aligned}
 \iint_{D_t}\left(  \varphi (u) | \underline{L}z|+ \varphi(\underline u)|L z|\right)(|L\psi_i||\underline L\psi_j||L\psi|+|L\psi_i|| L\psi_j||\underline L\psi|+ |L\psi_i||\underline L\psi_j||\underline L\psi|),
\end{aligned}
\end{equation*}
using Lemma \ref{lema1}, the estimates \eqref{supuesto} and \eqref{Llogalpha}, and the H\"{o}lder inequality, we get 
\begin{equation*}
\begin{aligned}
\iint_{D_t}\varphi (u) | \underline{L}z||L\psi_i||\underline L\psi_j||L\psi|&\lesssim \left(\iint_{D_t}\varphi(u) | \underline{L}z|^2|L\psi_i|^2 \right)^{1/2} \left(\iint_{D_t} \varphi(u) |\underline L\psi_j|^2|L\psi|^2 \right)^{1/2}\\
&\lesssim \left(\iint_{D_t}\frac{\delta^2 \varphi(u)|\underline L z|^2}{\varphi(\underline u)}\right)^{1/2}  \left(\iint_{D_t}\frac{\varepsilon^2 K_2^2\varphi(u) |\underline L\psi_j|^2}{\varphi^{3/2}(\underline u)} \right)^{1/2} \lesssim K\delta^3.
\end{aligned}
\end{equation*}
Finally, from the energy estimate, we can arrange all the previous estimates together, and for universal constants $C_j$ with $j\in \{ 3,4,...,9\}$, with $C_7,C_8,C_9\neq0,$ we have that for all $t\in [0,T^{*}]$:
\begin{equation*}
\mathcal E (t)+\mathcal F(t) \leq 2C_1C_0\delta^2 +C_3\varepsilon^2K^2\delta+C_4\varepsilon^3K^3\delta+C_5K^2\varepsilon^2\delta^2+C_6\varepsilon K \delta^2 +C_7 \delta^3+C_8K\delta^3+C_9\delta^4. 
\end{equation*}
Now, we take $\delta_0$ such that
\[ \delta_0 \leq \min\left\{\frac{C_0C_1}{2C_8K},\frac{C_1C_0}{2C_7} \right\}\quad \mbox{and} \quad  \delta_0^2 \leq \frac{C_1C_0}{2C_9}.
\] 
Then, we can see that for all $0< \delta < \delta_0$     and for all $t\in [0,T^*]$, we can take $0<\varepsilon <\varepsilon_0 $ such that $\varepsilon < \delta$ and 
\[ \max \left\{\varepsilon(C_3 K^2+C_6 K), \varepsilon^2( C_4K^3 +C_5K^2 )\right\} \leq \frac{C_1C_0}{2}. 
\]
 Therefore, we obtain
\begin{equation*}
\mathcal E (t)+\mathcal F(t) \leq  \frac{11}{2}C_1C_0 \delta^2,
\end{equation*}
improving the constant in \eqref{supuesto} and this finishes the proof of Theorem \ref{thm_aux}.
{\small
\bibliographystyle{unsrt}

}

\end{document}